\documentclass[a4paper,11pt,reqno,twoside]{amsart}

\usepackage{amsmath}
\usepackage{amsfonts}
\usepackage{amssymb}
\usepackage{amsthm}
\usepackage{mathtools}
\usepackage{color}
\usepackage{ifpdf}
\usepackage{array}
\usepackage{url}
\usepackage{multirow,bigdelim}
\usepackage{ascmac}
\usepackage[all]{xy} 
\usepackage{graphicx}

\numberwithin{equation}{section}

\newtheorem{theorem}{Theorem}
\newtheorem{lemma}{Lemma}
\newtheorem{proposition}{Proposition}

\newtheorem{remark}{Remark}

\newcommand*{\C}{\mathbb{C}}
\newcommand*{\R}{\mathbb{R}}
\newcommand*{\Q}{\mathbb{Q}}
\newcommand*{\Z}{\mathbb{Z}}

\newcommand{\comment}[1]{}
\title[Detecting zeros of Dirichlet $L$-functions]%
       {Detecting zeros of Dirichlet $L$-functions \\ via the Riemann zeta-function} 
\author[M. Suzuki]{Masatoshi Suzuki}
\subjclass[]{
11M26 
}
\keywords{Riemann Hypothesis, Generalized Riemann Hypothesis, 
vertical distribution of zeros
}
\AtBeginDocument{%
\begin{abstract}
Assuming the Riemann hypothesis, 
we show that a certain vertical distribution of 
the nontrivial zeros of the Riemann zeta-function is equivalent to 
the generalized Riemann hypothesis for Dirichlet $L$-functions. 
Furthermore, under both the Riemann hypothesis and the generalized Riemann hypothesis, 
we show that 
the nontrivial zeros of Dirichlet $L$-functions can be detected 
through those of the Riemann zeta-function.
\end{abstract}
\maketitle
}
\begin{document}

\section{Introduction}

We use the Landau symbols $O$ and $o$, as well as the Vinogradov symbol $\ll$, in their usual meanings. 
We denote the nontrivial zeros of the Riemann zeta-function $\zeta(s)$ by 
$\rho = \beta + i\gamma$ with $\beta, \gamma \in \R$. 
The Riemann Hypothesis (the RH) asserts that all such zeros satisfy $\beta = 1/2$. 
Similarly, the Generalized Riemann Hypothesis (the GRH) for a Dirichlet $L$-function $L(s,\chi)$ 
asserts that all nontrivial zeros 
$\rho_\chi = \beta_\chi + i\gamma_\chi$ satisfy $\beta_\chi = 1/2$. 
(For the definition of the nontrivial zeros of $L(s,\chi)$, 
see the review in the first paragraph of Section~\ref{section_1st_method}) 

Assuming RH, it is known that the distribution of the imaginary parts of the nontrivial zeros of $\zeta(s)$ contains information about the zeros of Dirichlet $L$-functions. 
This connection was first made precise by the Linnik--Sprind\v{z}uk theorem, 
which asserts that the GRH for all Dirichlet $L$-functions $L(s,\chi)$ 
associated with Dirichlet characters $\chi \bmod q$ $(\geq 3)$ 
is equivalent to the asymptotic formula
\begin{equation} \label{EQ_101}
\sum_{\gamma} e\left(\frac{\gamma}{2\pi} \log\frac{|\gamma|}{e}\right)
e^{-\frac{\pi}{2}|\gamma|}
\left( \frac{1}{x} + 2\pi i \frac{a}{q} \right)^{-1/2 - i\gamma}
= -\frac{\mu(q)}{\sqrt{2\pi} \varphi(q)}\, x + O(x^{1/2 + \varepsilon})
\end{equation}
as $x \to \infty$, for any $\varepsilon > 0$ and 
any integer $a$ with $1 \leq |a| \leq q/2$ and $(a, q) = 1$.
Here, the sum is taken over all imaginary parts $\gamma$ of the nontrivial zeros 
$\rho = 1/2 + i\gamma$ of $\zeta(s)$, counted with multiplicity, 
$\mu$ denotes the M\"obius function, $\varphi$ denotes the Euler totient function,  
and, as usual, 
\[
e(x) := \exp(2\pi i x).
\]
This result was established by Sprind\v{z}uk~\cite{Sp75, Sp76}, 
building on an idea of Linnik~\cite{Li47}, 
and was later generalized by Fujii~\cite{Fu88a, Fu88b} and Banks~\cite{Ba23, Ba24a, Ba24b}.  
Suzuki~\cite{Su04} extended it to functions in the Selberg class, 
and Kaczorowski and Perelli~\cite{KP08} 
presented a different form of the Linnik--Sprind\v{z}uk phenomenon 
within the Selberg class.

On the other hand, Fujii~\cite{Fu88b, Fu92} proved a result similar to 
\eqref{EQ_101} by replacing the infinite sum on the left-hand side with 
a finite one. Specifically, he showed that 
the GRH  for all Dirichlet $L$-functions associated with Dirichlet characters 
modulo $q$ $(\geq 3)$ is equivalent to the asymptotic formula
\begin{equation} \label{EQ_102}
\sum_{0<\gamma\leq x} e\left(\frac{\gamma}{2\pi} \log \frac{\gamma}{2\pi (a/q) e}\right)
= - \frac{e^{\pi i/4}\mu(q)}{2\pi\varphi(q)\sqrt{a/q}}\,x
+ O\left(x^{1/2+\varepsilon}\right)
\end{equation}
as $x \to \infty$, for any integer $a$ with $1 \leq a \leq q$ and $(a,q) = 1$ 
(\cite[Corollary~3]{Fu92}). 

The main results of this paper extend the Linnik--Sprind\v{z}uk theorem and Fujii's result  
in the sense that the vertical distribution of the nontrivial zeros of $\zeta(s)$ 
makes it possible to detect the ordinates of the nontrivial zeros of Dirichlet $L$-functions. 
To state these results, we define
\begin{equation} \label{EQ_103}
Z(x;t,\alpha)
:= 
\sum_{0<\gamma - t \leq 2\pi \alpha x} 
\frac{1}{\sqrt{\gamma-t}}\,
e\!\left(\frac{\gamma-t}{2\pi}\log \frac{\gamma-t}{2\pi \alpha e} \right)  
\left(1 - \frac{1}{\log x}\log\frac{\gamma-t}{2\pi \alpha}\right)
\end{equation}
assuming RH, 
where the sum is taken over all imaginary parts $\gamma$ 
of the nontrivial zeros $\rho = 1/2 + i\gamma$ of $\zeta(s)$, counted with multiplicity.
Note that the mollifier 
$1-(\log x)^{-1}\log((\gamma-t)/(2\pi \alpha))$ 
in definition~\eqref{EQ_103} 
is consistent with the summation range $0 < \gamma - t \leq 2\pi \alpha x$.

\begin{theorem} \label{thm_1} 
Assume RH. 
Let $t$ be a real number that is not the ordinate of any nontrivial zero of $\zeta(s)$, 
and let $q \geq 3$ be a positive integer. 
Let $m(t,\chi)$ denote the order of the Dirichlet $L$-function $L(s,\chi)$ at the point $s = 1/2 + it$. 
Then the GRH for all Dirichlet $L$-functions $L(s,\chi)$ associated with 
Dirichlet characters $\chi \bmod q$ 
is equivalent to the asymptotic formula
\begin{equation} \label{EQ_104}
Z(x;t,a/q_1)
 = -\frac{e^{\pi i/4}\mu(q_1)}{\sqrt{2\pi}\,\varphi(q_1)} \cdot 
   \frac{1}{(1/2 - it)^2} \cdot \frac{x^{1/2 - it}}{\log x}
   + O(x^\varepsilon) 
\end{equation}
as $x \to \infty$, for any $\varepsilon>0$ and 
for any integer $a$ with $1 \leq a < q_1$, $(a,q_1) = 1$, and $q_1 \mid q$,  
where the implied constant may depend on $t$, $q_1$, and $\varepsilon$.

Moreover, assuming the GRH for all $L(s,\chi)$ 
with $\chi \bmod q$, we obtain the formula
\begin{equation} \label{EQ_105} 
\aligned  
 Z(x;t,a/q_1) 
& = -\frac{e^{\pi i/4}\mu(q_1)}{\sqrt{2\pi}\,\varphi(q_1)} \cdot 
    \frac{1}{(1/2 - it)^2} \cdot \frac{x^{1/2 - it}}{\log x}\\
& \quad +\frac{e^{\pi i/4}}{2\sqrt{2\pi}\,\varphi(q_1)} 
    \sum_{\chi \bmod q_1} \overline{\tau(\chi)}\,\chi(a)\,m(t, \chi)
    \log x 
    + O(1)
\endaligned 
\end{equation}
as $x \to \infty$, for any integer $a$ with $1 \leq a < q_1$, $(a,q_1) = 1$, and $q_1 \mid q$,  
where the sum $\sum_{\chi \bmod q_1}$ is taken over all 
Dirichlet characters $\chi\bmod q_1$, 
$\tau(\chi)$ denotes the Gauss sum for $\chi$, and  
the implied constant may depend on $t$ and $q_1$. 
\end{theorem}

\begin{remark}
By the precise definition of GRH stated 
in the first paragraph of Section~\ref{section_1st_method}, 
if the GRH holds for all $\chi \bmod q$, 
then it also holds for all $\chi \bmod q_1$ 
for every $q_1 \mid q$.
\end{remark}

If the aim is to establish an equivalent condition to GRH, 
it suffices to consider the case $t=0$ in Theorem \ref{thm_1}.  
However, allowing $t$ to vary (almost) arbitrarily 
enables 
the detection of the nontrivial zeros of Dirichlet $L$-functions, 
as in the second assertion of Theorem \ref{thm_1}.
Using the orthogonality of Dirichlet characters, 
Theorem~\ref{thm_1} can be decomposed into statements   
concerning individual Dirichlet $L$-functions.

\begin{theorem} \label{thm_0724}
Assume RH, and let $t$ be a real number that is not the ordinate of any nontrivial zero of $\zeta(s)$. 
Let $\chi$ be a non-principal 
Dirichlet character modulo $q$ ($\geq 3$) such that $\tau(\chi) \ne 0$, 
and let $m(t,\chi)$ denote the order of 
the Dirichlet $L$-function $L(s,\chi)$ at $s = 1/2 + it$. 
Then the GRH for $L(s,\chi)$ is equivalent to the bound
\begin{equation} \label{EQ_106}
\sum_{a=1}^{q} \overline{\chi}(a) Z(x;t,a/q) = O(x^\varepsilon) 
\end{equation}
as $x \to \infty$, for any $\varepsilon > 0$, 
where the implied constant may depend on $t$, $q$, and $\varepsilon$.

Moreover, assuming the GRH for $L(s,\chi)$, we have 
\begin{equation} \label{EQ_107}
\sum_{a=1}^{q} \overline{\chi}(a)Z(x;t,a/q) 
 = \frac{e^{\pi i/4}\,\overline{\tau(\chi)}}{2\sqrt{2\pi}} \, m(t, \chi) \log x
+ O(1)
\end{equation}
as $x \to \infty$, 
where the implied constant may depend on $t$ and $q$. 

On the other hand, for the principal character $\chi_0 \bmod q$, 
\begin{equation} \label{EQ_108}
\begin{aligned}
\sum_{a=1}^{q} \chi_0(a)Z(x;t,a/q) 
& =  -\frac{e^{\pi i/4}\mu(q)}{\sqrt{2\pi}}\cdot 
 \frac{1}{(1/2 - it)^2} \cdot \frac{x^{1/2 - it}}{\log x} \\ 
& \quad + A_0(t,a/q)
+ O\left( \frac{1}{\log x} \right)
\end{aligned}
\end{equation}
as $x \to \infty$ for some constant $A_0(t,a/q)$, 
where the implied constant may depend on $t$ and $q$. 
\end{theorem}
\begin{remark} 
If $\chi \bmod q$ is primitive, then $\tau(\chi)\not=0$, since $|\tau(\chi)|=\sqrt{q}$. 
If the imprimitive character $\chi \bmod q$ is induced by a primitive character $\chi^\ast \bmod q^\ast$, then
\[
\tau(\chi) = \mu(q/q^\ast) \chi^\ast(q/q^\ast)\tau(\chi^\ast)
\]
(\cite[Theorem 9.10]{MV07}). 
Hence, if $q/q^\ast$ is square-free and coprime to $q^\ast$, we have $\tau(\chi) \ne 0$.
\end{remark}
\begin{figure}
 \begin{minipage}{15cm}
 \centering
\includegraphics[width=10cm]{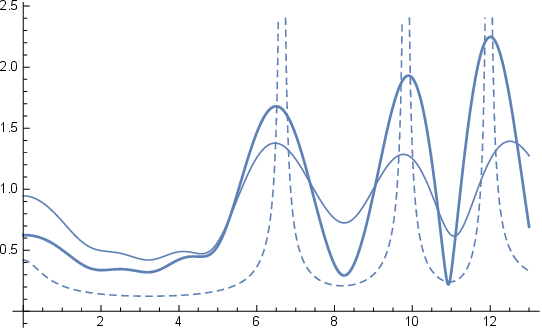}
\caption{Plot of $\vert \sum_{a=1}^{5}\bar{\chi}(a) Z(100;t,a/5) \vert$ in the range $0 \leq t \leq 13$ 
for the real primitive character $\chi \bmod 5$. 
The thin solid line represents $\vert (2\pi)^{-1/2} \sqrt{5} 
\sum_{a=1}^{5}\bar{\chi}(a) \Psi(100;1/2+it,a/5) \vert$, 
and the thin dashed line represents $\vert L'/L(1/2+it,\chi) \vert$. 
(cf. Proposition \ref{prop_3}) 
}
\label{fig_01}
  \end{minipage} 
\end{figure}

The relations \eqref{EQ_105}, \eqref{EQ_107}, and \eqref{EQ_108} 
can be interpreted as stating that the \emph{local} distribution of the nontrivial zeros 
of the Riemann zeta-function in the interval $[t, t+2\pi \alpha x]$ 
encodes information about the distribution of the nontrivial zeros of Dirichlet $L$-functions 
in the same interval.
Figure~\ref{fig_01} shows a plot of the absolute value of the left-hand side of~\eqref{EQ_107} for $x = 100$ and $0 \leq t \leq 13$, in the case of the real primitive Dirichlet character $\chi \bmod 5$. We observe that the graph exhibits peaks at the ordinates of the nontrivial zeros of $L(s, \chi)$, as expected from \eqref{EQ_107}. 

To extend Theorems \ref{thm_1} and \ref{thm_0724} to the case $t = \gamma$, the statements of the results become more complicated, and the proofs require somewhat more involved calculations than those presented later. Therefore, we do not address this extension in this paper.

On the other hand, as we outline in Section~\ref{section_6},  
it is possible to interchange the roles of  
the Riemann zeta-function and Dirichlet $L$-functions in \eqref{EQ_107}.  
That is, one can detect the nontrivial zeros of the Riemann zeta-function  
through those of Dirichlet $L$-functions.  
In this way, the nontrivial zeros of the Riemann zeta-function and  
those of Dirichlet $L$-functions are in a reciprocal relation.  
\medskip

A result similar to the first half of Theorem~\ref{thm_0724} 
was already implicitly stated in Fujii~\cite{Fu92}. In fact, in the proof of Theorem 2 in~\cite{Fu88b}, 
he discusses a reduction of~\eqref{EQ_102} to individual $L$-functions. 

Fujii also clarified that rationals and irrationals play distinct roles in the vertical distribution of the nontrivial zeros of $\zeta(s)$ by proving that
\[
\aligned 
\lim_{T\to\infty} & \frac{1}{T}
\sum_{0<\gamma\leq T} e\left(\frac{\gamma}{2\pi} \log \frac{\gamma}{2\pi\alpha e}\right) \\
& = 
\begin{cases}
~ \displaystyle{-\frac{e^{\pi i/4}}{2\pi} \frac{\mu(q)}{\varphi(q)}\frac{1}{\sqrt{\alpha}} }
& \text{if $\alpha=a/q$ with integers $a$ and $q \geq 1$, $(a,q)=1$}, \\
~0 & \text{otherwise}, 
\end{cases}
\endaligned 
\]
under RH (\cite[Corollary 5]{Fu82}). 
An analogue of this result for the sum in Theorem~\ref{thm_1} can be obtained as follows.

\begin{theorem} \label{thm_3} Assume RH. 
Then, for any $t \in \R$, 
\[
\aligned  
Z(x;t,\alpha)
& = 
\begin{cases}
~\displaystyle{
-\frac{e^{\pi i/4}\mu(q)}{\sqrt{2\pi}\varphi(q)}\cdot 
 \frac{1}{(1/2 - it)^2} \cdot \frac{x^{1/2 - it}}{\log x}\, (1+o(1))} & \\[10pt]
\hspace{88pt}  \text{if $\alpha=a/q$ with integers $a$ and $q \geq 1$, $(a,q)=1$}, & \\[10pt]
~\displaystyle{
o(x^{1/2}/\log x)
} \qquad \text{otherwise}, & 
\end{cases}
\endaligned 
\]
as $x \to \infty$, where the implied constant may depend on $t$ and $\alpha$. 
\end{theorem}

The proofs of the main results, Theorems~\ref{thm_1}, \ref{thm_0724},  
and~\ref{thm_3}, are based on a method different from Fujii's~\cite{Fu92}.  
They rely on evaluating the sum
\begin{equation} \label{EQ_109}
\Psi(x; s, \alpha)
:=
\sum_{n \leq x} \frac{\Lambda(n)e(-\alpha n)}{n^s}
\left( 1 - \frac{\log n}{\log x} \right)
\end{equation}
in two different ways, where $x > 1$, $s \in \C$ with $\Re(s) > 0$, and $\alpha >0$.  
(For the exponential factor $e(-\alpha n)$ to be meaningful, we require $\alpha \neq 0$.  
In that case, we may assume without loss of generality that $\alpha > 0$ by considering complex conjugation.)
In the first approach, if $\alpha$ is rational, the orthogonality of Dirichlet characters  
is used to express $\Psi(x; s, \alpha)$ in terms of Dirichlet $L$-functions  
(Section~\ref{section_1st_method}).  
If $\alpha$ is irrational, the sum is estimated by reducing it to a result of Vinogradov  
on $\sum_{n \leq x} \Lambda(n) e(n\alpha)$ (Section~\ref{section_5}).
In the second approach, the explicit formula for $\zeta(s)$ is applied to express $\Psi(x; s, \alpha)$  
in terms of the nontrivial zeros of $\zeta(s)$ (Section~\ref{section_4}),  
with the help of auxiliary results on exponential integrals in Section~\ref{section_3}.  
By combining the results obtained from these two approaches,  
we derive the main results stated above.  
Furthermore, in Section~\ref{section_6}, we show that similar results hold  
even when the left-hand sides of \eqref{EQ_105} and \eqref{EQ_107}  
are replaced by sums over the nontrivial zeros of Dirichlet $L$-functions.

\section{Formulas for sums via Dirichlet $L$-functions} \label{section_1st_method}

We begin by precisely stating the definition of the nontrivial zeros of Dirichlet $L$-functions.  
For a primitive Dirichlet character $\chi \bmod q$ ($\geq 3$),  
we set $\kappa = \kappa(\chi) = 0$ if $\chi(-1) = 1$, 
and $\kappa = \kappa(\chi) = 1$ if $\chi(-1) = -1$, and define
\[
\xi(s,\chi) := \left(\frac{q}{\pi}\right)^{(s+\kappa)/2}
\Gamma\left(\frac{s+\kappa}{2}\right) L(s,\chi).
\]
Then $\xi(s,\chi)$ is an entire function satisfying 
$\xi(s,\chi) = \varepsilon(\chi)\,\xi(1-s,\overline{\chi})$ 
for some constant $\varepsilon(\chi)$ with $|\varepsilon(\chi)| = 1$  
(\cite[(10.17), Corollary 10.8]{MV07}).  
The zeros of $L(s,\chi)$ in the critical strip $0 \leq \Re(s) \leq 1$ with $s \neq 0$  
are called \emph{nontrivial zeros} and coincide with the zeros of $\xi(s,\chi)$.  
If $\chi \bmod q$ is imprimitive,  
let $\chi^\ast \bmod q^\ast$ be the primitive character that induces $\chi$. 
We regard the nontrivial zeros of $L(s,\chi)$ as those of $L(s,\chi^\ast)$.  
Hence $L(s,\chi)$ satisfies the GRH if $L(s,\chi^\ast)$ does.
\medskip

In this section, we first evaluate \eqref{EQ_109} for $s \in \C$ with $\Re(s) > 0$  
using the orthogonality of Dirichlet characters.  
By specializing this to $s = 1/2 + it$, we obtain 
the following proposition.

\begin{proposition} \label{prop_1}
Let $q \in \Z_{>0}$ and $t \in \R$.  
Then the GRH for all Dirichlet $L$-functions $L(s,\chi)$  
with $\chi \bmod q$, including the principal character,  
holds if and only if
\begin{equation} \label{EQ_201}
\Psi(x; 1/2 + it, a/q_1) 
- \frac{\mu(q_1)}{\varphi(q_1)} \cdot 
\frac{x^{1/2 - it}}{(1/2 - it)^2 \log x} 
= O(x^\varepsilon) 
\end{equation}
as $x \to \infty$, for any $\varepsilon>0$ and 
for any integer $a$ with $1 \leq a < q_1$, $(a,q_1) = 1$, and $q_1 \mid q$,  
where the implied constant may depend on $t$, $q$, and $\varepsilon > 0$. 
\end{proposition}

\begin{proof} 
Let $s \in \C$ with $\Re(s)>0$. 
By splitting the sum in definition~\eqref{EQ_109} into the terms 
with $n$ coprime to $q_1$ and those that are not, we obtain
\begin{equation} \label{EQ_202}
\aligned 
\Psi(x; s, a/q_1) 
& = 
\sum_{\substack{n \leq x \\ (n,q_1)=1}}
 \frac{\Lambda(n)e(-an/q_1)}{n^s}\left( 1 - \frac{\log n}{\log x} \right) \\
& \quad 
+ \sum_{p \mid q_1} \log p
\sum_{m=1}^{\lfloor \log x/\log p \rfloor} \frac{e(-ap^m/q_1)}{p^{ms}}\left( 1 - \frac{\log p^m}{\log x} \right).
\endaligned 
\end{equation}
If $(an,q_1)=1$, then we have
\begin{equation} \label{EQ_203}
\aligned 
 e(-an/q_1)
& =\frac{1}{\varphi(q_1)} \sum_{\chi \bmod q_1} \overline{\tau(\chi)}\chi(a)\chi(n),
\endaligned 
\end{equation}
where the sum is taken over all Dirichlet characters modulo $q_1$ 
(\cite[Exercises 9.2.1.1.(a)]{MV07}). 

Using~\eqref{EQ_203}, the first sum on the right-hand side of~\eqref{EQ_202} can be evaluated as
\[
\aligned 
\sum_{\substack{n \leq x \\ (n,q_1)=1}} 
 \frac{\Lambda(n)e(-an/q_1)}{n^s}\left( 1 - \frac{\log n}{\log x} \right)  
& = 
\frac{1}{\varphi(q_1)}\sum_\chi \overline{\tau(\chi)}\,\chi(a)
\Psi(x; s, \chi), 
\endaligned 
\]
where we define
\begin{equation} \label{EQ_204}
\Psi(x; s, \chi):=
\sum_{n \leq x} 
 \frac{\Lambda(n)\chi(n)}{n^s}\left( 1 - \frac{\log n}{\log x} \right). 
\end{equation}
Here we used the fact that $\chi(n) = 0$ whenever $(n, q_1) > 1$, 
so the condition $(n,q_1)=1$ may be dropped from the inner sum \eqref{EQ_204}.

As for the second sum on the right-hand side of~\eqref{EQ_202}, 
the inner sum over $m$ can be calculated as follows
\[
\aligned 
\,&\sum_{m=1}^{\lfloor \log x/\log p \rfloor} 
\frac{e(-ap^m/q_1)}{p^{sm}}\left( 1 - \frac{\log p^m}{\log x} \right) \\
& = \frac{1}{\log x} 
\sum_{m=1}^{\lfloor \log x/\log p \rfloor} 
\frac{e(-ap^m/q_1)}{p^{sm}} \int_{p^m}^{x} \frac{1}{t} \, dt 
= \frac{1}{\log x} \int_{1}^{x}\frac{1}{t}
\sum_{m=1}^{\lfloor \log t/\log p \rfloor} 
\frac{e(-ap^m/q_1)}{p^{sm}}   \, dt \\
& = \frac{1}{\log x} \int_{1}^{x}\frac{1}{t}
\left(\sum_{m=1}^{\infty} 
\frac{e(-ap^m/q_1)}{p^{sm}}  + O(t^{-\Re(s)}) \right) \, dt 
 = \sum_{m=1}^{\infty} 
\frac{e(-ap^m/q_1)}{p^{sm}} + O\left(\frac{1}{\log x}\right), 
\endaligned
\]
where the sum on the right-hand side is absolutely convergent, 
and the implied constant depends only on $\Re(s)>0$.  

Therefore, we obtain
\begin{equation} \label{EQ_205}
\aligned 
\Psi(x; s, a/q_1) 
& = 
\frac{1}{\varphi(q_1)}\sum_{\chi \bmod q_1} \overline{\tau(\chi)}\,\chi(a)
\Psi(x; s, \chi)
+ A_1(s,a/q_1)+ O\left(\frac{1}{\log x}\right)
\endaligned 
\end{equation}
as $x \to \infty$ for some constant $A_1(s,a/q_1)$. 
Multiplying both sides of \eqref{EQ_205} 
by $\bar{\chi}(a)$ and summing over $a$ with $1 \leq a < q_1$,  
we obtain 
\begin{equation} \label{EQ_206}
\sum_{a=1}^{q_1} \bar{\chi}(a)\Psi(x; s, a/q_1) 
= \overline{\tau(\chi)}\Psi(x; s, \chi)+ A_2(s,a/q_1)+ O\left(\frac{1}{\log x}\right)
\end{equation}
by the orthogonality of Dirichlet characters: 
\[
\frac{1}{\varphi(q)}\sum_{a=1}^q \chi(a) 
= 
\begin{cases}
1 & \text{if $\chi=\chi_0$}, \\[6pt]
0 & \text{if $\chi\ne\chi_0$}. 
\end{cases}
\]
Therefore, if $\chi$ is a primitive Dirichlet character modulo $q$ 
and \eqref{EQ_201} holds for all $1 \leq a < q_1$ with $s = 1/2 + it$ and $q_1 = q$, 
then $\Psi(x; 1/2 + it, \chi) \ll_\varepsilon x^\varepsilon$ by \eqref{EQ_206}. 
The estimate implies the GRH for $L(s,\chi)$ via the integral formula
\begin{equation} \label{EQ_207}
\int_{1}^{\infty}  
\Bigl(-\Psi(x;s,\chi)\log x\Bigr) \, x^{-z+1/2} \, \frac{dx}{x} 
= \frac{1}{(z-1/2)^2} \frac{L^\prime}{L}\left(s+z-\frac{1}{2},\chi\right) 
\end{equation}
for $s=1/2+it$, 
which can be proved by an argument similar to that in \cite[(4.5)]{Su24}. 
If $\chi$ is imprimitive, let $\chi^\ast$ be the primitive character modulo $q^\ast$ that induces $\chi$. 
Suppose that \eqref{EQ_201} holds for all $1 \leq a < q^\ast$ with $s = 1/2 + it$ and $q_1 = q^\ast$. 
Then, by applying \eqref{EQ_206} with $q_1 = q^\ast$, the GRH for $L(s, \chi^\ast)$ follows from the above argument, and hence the GRH also holds for $L(s, \chi)$. 

Conversely, by inverting the integral formula \eqref{EQ_207}, 
the GRH for $L(s,\chi)$ yields the estimate 
$\Psi(x; 1/2 + it, \chi) \ll_\varepsilon x^\varepsilon.$ 
Therefore, assuming the GRH for $L(s,\chi)$ for all $\chi\bmod q_1$, 
we obtain \eqref{EQ_201} from \eqref{EQ_205}.
\end{proof} 

An explicit formula for $\Psi(x; s, \chi)$ 
follows from the inversion of \eqref{EQ_207}. 
For a non-principal character $\chi$, let $\chi^\ast$ be the primitive character that induces $\chi$. Dividing both sides of \cite[(4.7)]{Su24} by $\log x$ gives
\begin{equation} \label{EQ_208} 
\aligned 
\Psi(x;s,\chi)
& = -\frac{L'}{L}(s,\chi) 
- \frac{1}{\log x}\left(\frac{L'}{L}\right)^\prime(s,\chi) \\
& \quad 
- \frac{1}{\log x}\sum_{\rho_\chi} \frac{x^{\rho_\chi-s}}{(\rho_\chi-s)^2} 
- \frac{1}{\log x}\sum_{k=0}^{\infty} \frac{x^{-2k-\kappa(\chi^\ast)-s}}{(2k+\kappa(\chi^\ast)+s)^2}
\endaligned 
\end{equation}
for $x>1$ and $s \ne \rho_\chi, -2k - \kappa(\chi^\ast)$ ($k \in \Z_{\geq 0}$),  
where the sum $\sum_{\rho_\chi}$ is taken over all nontrivial zeros of $L(s,\chi)$, 
counted with multiplicity (cf.~the first paragraph of this section). 
In \eqref{EQ_208}, the sum over $\rho_\chi$ converges absolutely, 
since $\sum_{\rho_\chi} |\rho_\chi|^{-1-\delta} <\infty$ 
for any $\delta>0$. 

If $s$ is a zero of $L(s,\chi)$ of order $m_s$, then we have
\begin{equation} \label{EQ_209}
\aligned 
\Psi(x;s,\chi)
& = -\frac{m_s}{2}\log x 
-\lim_{z \to s} \left[\frac{L^\prime}{L}(z,\chi) - \frac{m_s}{z-s}\right] \\
& \quad 
- \frac{1}{\log x} 
\lim_{z \to s} \left[\left(\frac{L^\prime}{L}\right)^\prime(z,\chi) + \frac{m_s}{(z-s)^2}\right] \\
& \quad
- \frac{1}{\log x} \sum_{\rho_\chi\ne s} \frac{x^{\rho_\chi-s}}{(\rho_\chi-s)^2} 
- \frac{1}{\log x} \sum_{k=0}^{\infty} \frac{x^{-2k-\kappa(\chi^\ast)-s}}{(2k+\kappa(\chi^\ast)+s)^2}
\endaligned 
\end{equation}
by \cite[(4.9)]{Su24}. 
On the other hand, if $\chi$ is the principal character $\chi_0$, then  
\begin{equation} \label{EQ_210} 
\aligned 
\Psi(x;s,\chi_0)
= 
\Psi(x;s,1)
- \sum_{p \mid q} \frac{\log p}{p^s - 1}
+ O_q\!\left(\frac{1}{\log x}\right).
\endaligned 
\end{equation}
Moreover, from \cite[Lemma~1]{So09} we have
\begin{equation} \label{EQ_211} 
\aligned 
\Psi(x;s,1)
& = \frac{x^{1-s}}{(1-s)^2\log x} - \frac{\zeta'}{\zeta}(s)
- \frac{1}{\log x}\left(\frac{\zeta'}{\zeta}\right)^\prime(s) \\
& \quad 
- \frac{1}{\log x}\sum_{\rho} \frac{x^{\rho-s}}{(\rho-s)^2} 
- \frac{1}{\log x}\sum_{k=1}^{\infty} \frac{x^{-2k-s}}{(2k+s)^2}
\endaligned 
\end{equation}
for $x > 1$ and $s \ne \rho, -2k$ ($k \in \Z_{> 0}$),  
where the sum $\sum_\rho$ is taken over all nontrivial zeros of $\zeta(s)$, counted with multiplicity,  
and converges absolutely since  
\begin{equation} \label{EQ_212}
\sum_\rho |\rho|^{-1-\delta} < \infty
\end{equation}
for any $\delta > 0$. 

Recalling that $\tau(\chi_0) = \mu(q)$ (\cite[(4.7), (9.5)]{MV07}),  
we substitute \eqref{EQ_208}, \eqref{EQ_209}, \eqref{EQ_210}, and \eqref{EQ_211} into \eqref{EQ_205},  
from which the following proposition follows.

\begin{proposition} \label{prop_2}
Let $a$ and $q$ be integers with $q \geq 1$ and $(a, q) = 1$. 
Assume the GRH for all $L(s,\chi)$ with Dirichlet characters $\chi \bmod q$. 
\begin{enumerate}
\item 
If either $\Re(s) > 1/2$, or $\Re(s) = 1/2$ and $L(s,\chi) \ne 0$ 
for all Dirichlet characters $\chi \bmod q$, 
including the principal character $\chi_0$, 
then
\begin{equation} \label{EQ_213}
\aligned 
\Psi(x;s,a/q) 
& = \frac{\mu(q)}{\varphi(q)} \cdot 
 \frac{x^{1-s}}{(1-s)^2\log x} \\
& \quad - \frac{1}{\varphi(q)} \left[ \mu(q)
\left( \frac{\zeta'}{\zeta}(s) 
+ \sum_{p \mid q} \frac{\log p}{p^s-1} \right) \right. \\
& \qquad \qquad \qquad \quad \left. +\sum_{\chi\not=\chi_0} \overline{\tau(\chi)}\,\chi(a)
\frac{L'}{L}(\chi, s) \right] \\ 
& \quad + \sum_{p \mid q} \log p\sum_{m=1}^{\infty} 
\frac{e(-ap^m/q)}{p^{sm}} 
 + O\left(\frac{1}{\log x}\right)
\endaligned 
\end{equation}
as $x \to \infty$, 
where the implied constant may depend on $s$ and $q$.
This can also be written as
\begin{equation} \label{EQ_214} 
\aligned 
\Psi(x;s,a/q) 
- \frac{\mu(q)}{\varphi(q)} \cdot 
 \frac{x^{1-s}}{(1-s)^2\log x} 
 =  A_3(s,a/q)+ O\left(\frac{1}{\log x}\right)
\endaligned 
\end{equation}
for some constant $A_3(s,a/q)$. 

\item If $s = 1/2 + it$ for a real number $t$ that is not the ordinate of any nontrivial zero of $\zeta(s)$, 
and $m(t, \chi)$ denotes the order of $L(s, \chi)$ at $s = 1/2 + it$, then 
\begin{equation} \label{EQ_0215} 
\aligned 
\Psi(x;s,a/q) 
& - \frac{\mu(q)}{\varphi(q)} \cdot 
 \frac{x^{1-s}}{(1-s)^2\log x} \\
& =  - \frac{1}{2\varphi(q)}\sum_\chi \overline{\tau(\chi)}\,\chi(a) \, m(t,\chi) \log x \\
& \quad 
 + A_4(s,a/q)
 + O\left(\frac{1}{\log x}\right)
\endaligned 
\end{equation}
as $x \to \infty$, for some constant $A_4(s,q/q)$, 
where the implied constant may depend on $s$ and $q$.
\end{enumerate}
\end{proposition}

The infinite sum appearing on the right-hand side of \eqref{EQ_213}, 
$\sum_{m=1}^{\infty} e(-ap^m/q)p^{-sm}$, 
can be rewritten as a finite sum. 
However, we omit the details, since it will not be needed later.

\section{Auxiliary results on certain exponential integrals} \label{section_3} 

In this section, 
we consider the exponential integrals
\[
\aligned
G(x;\alpha,z)
: &= \int_{1}^{x} e(-\alpha u) \left(1-\frac{\log u}{\log x}\right) u^z \, du 
\endaligned
\]
for $x > 1$, $\alpha > 0$, and $z \in \mathbb{C}\) with \(\Re(z) > -1$, 
and state formulas and estimates that will be needed later 
for calculating the sum $\Psi(x;s,\alpha)$ in \eqref{EQ_109} 
by a method different from that in Section \ref{section_1st_method}. 
First, we recall the following well-known lemma.

\begin{lemma} \label{lem_1}
Let $f(x)$ and $g(x)$ be real functions, 
$f(x)$ is differentiable, 
$g(x)/f'(x)$ monotonic, 
and $f'(x)/g(x) \geq m >0$, 
or $f'(x)/g(x) \leq -m<0$,  
throughout the interval $[a,b]$. 
Then 
\[
\left|
\int_{a}^{b} g(x)e^{if(x)} \, dx
\right| \leq \frac{4}{m}.
\]
\end{lemma}
\begin{proof}
See~\cite[Lemma 4.3]{Tit86}. 
\end{proof}

Let $E_s(z)$ denote the generalized exponential integral, defined by  
\[
E_s(z) := \frac{e^{-z}}{\Gamma(s)}
\int_{0}^{\infty} \frac{e^{-zt} t^{s-1}}{1 + t} \, dt
\]
for $\Re(s) > 0$ and $|\arg z| < \pi/2$,  
and extended to other values of $s$ and $z$ by analytic continuation 
(\cite[(2.11.5)]{NIST}).

\begin{lemma} \label{lem_2}
Let $\alpha>0$. For $-1 < z \leq 1/2$, we have 
\begin{equation} \label{EQ_301}
G(x;\alpha,z)
=
E_{-z}(2\pi i \alpha) + O\left( \frac{1}{\log x} \right)
\end{equation}
as $x \to \infty$, 
where the implied constant may depend on $\alpha$ and $z$. 
\end{lemma} 

\begin{proof}

We begin with $z \in \C$ satisfying $\Re(z) > -1$, 
and later restrict $z$ to real values in the range $-1< z \leq 1/2$. 
We put $g(x;\alpha,s):=G(x;\alpha,s-1)$. Then we have 
\[
\aligned 
g(x;\alpha,s)
& = 
\int_{1}^{x} e(-\alpha u) \, u^{s-1} \, du 
-  \frac{1}{\log x}\frac{d}{ds}
\int_{1}^{x} e(-\alpha u) \, u^{s-1} \, du
\endaligned 
\]
for $\Re(s)>0$. The integral on the right-hand side is expressed as 
\begin{equation} \label{EQ_302}
\int_{1}^{x} e(-\alpha u)  \, u^{s-1} \, du 
 = (2 \pi i\alpha)^{-s} \gamma(s, 2\pi i\alpha x)-(2 \pi i\alpha)^{-s} \gamma(s, 2\pi i\alpha)
\end{equation} 
by using the lower incomplete gamma function
\[
\gamma(s,z) = \int_{0}^{z} e^{-t} \, t^{s-1} \, dt, \quad \Re(s)>0, 
\]
since it holds that
\begin{align*}
\int_{0}^{x} e(-\alpha u)  \, u^{s-1} \, du
& = (2\pi i\alpha)^{-s} \int_{0}^{2 \pi i \alpha x} e^{-t} t^s \frac{dt}{t} 
 = (2 \pi i\alpha)^{-s} \gamma(s, 2\pi i\alpha x)
\end{align*}
if $\Re(s)>0$ by \cite[8.2 (i) Definitions]{NIST}. 
Noting that \( \Gamma(s,z) = \Gamma(s) - \gamma(s,z) \), and using \cite[(29)]{GGMS90}, 
we have
\begin{equation} \label{EQ_303}
\frac{d}{ds} \gamma(s,z)  
= \frac{d}{ds} \Gamma(s) - \Gamma(s) \log z + \gamma(s,z) \log z - z T(3,s,z),
\end{equation}
where $T(3,s,z)$ is expressed via the Meijer \( G \)-function as
\[
T(3,s,z) = G_{2,3}^{3,0}\left(\left.\begin{matrix} 0,\,0 \\ s-1,\,-1,\,-1 \end{matrix} \right|z\right)
\]
(\cite[(31)]{GGMS90}). 
Using \eqref{EQ_302}, \eqref{EQ_303}, 
and setting $\beta = 2\pi \alpha$, we obtain
\[
\aligned 
g(x;\alpha,s) 
& =  (i\beta)^{-s} \Gamma(s,i\beta) \\
& \quad 
+ \frac{1}{\log x} (i\beta)^{-s} \left[ (i\beta x) T(3,s,i\beta x) - (i\beta) T(3,s,i\beta) \right] \\
& = E_{1-s}(i\beta) \\
& \quad 
+ \frac{1}{\log x} (i\beta)^{-s} \left[ 
G_{2,3}^{3,0}\left(\left.\begin{matrix} 1,\,1 \\ 0,\,0,\,s \end{matrix} \right| i\beta x \right)
- 
G_{2,3}^{3,0}\left(\left.\begin{matrix} 1,\,1 \\ 0,\,0,\,s \end{matrix} \right| i\beta \right)
\right].
\endaligned 
\]
Here, we have used 
\begin{equation} \label{EQ_304}
E_s(w) = w^{s-1}\Gamma(1-s,w)
\end{equation}
(\cite[(8.19.1)]{NIST}) and 
$\displaystyle{
z T(3,s,z) = G_{2,3}^{3,0}\left(\left.\begin{matrix} 1,\,1 \\ 0,\,0,\,s \end{matrix} \right|z\right)
}$
(\cite[(16.19.2)]{NIST}).
Therefore,
\begin{equation} \label{EQ_305}
\aligned 
G(x;\alpha,z)
& =E_{-z}(2 \pi i\alpha) \\
& \quad  + \frac{1}{\log x} (2\pi i\alpha)^{-1-z} \left[ 
G_{2,3}^{3,0}\left(\left.\begin{matrix} 1,\,1 \\ 0,\,0,\,1+z \end{matrix} \right| 2 \pi i\alpha x \right) 
\right. \\
& \qquad \qquad \qquad \qquad \qquad \qquad 
\left. 
-
G_{2,3}^{3,0}\left(\left.\begin{matrix} 1,\,1 \\ 0,\,0,\,1+z \end{matrix} \right| 2 \pi i\alpha \right)
\right]
\endaligned 
\end{equation}
for $z \in \C$ with $\Re(z)>-1$.
\medskip

In the following, we restrict \( z \) to real values in the range $-1 < z \leq 1/2 $.
For $z \geq -1$, $\beta > 0$, and $x > 1$, we have the integral formula
\begin{equation} \label{EQ_306}
G_{2,3}^{3,0}\!\left(\left.\begin{matrix} 1,\,1 \\[2pt] 0,\,0,\,1 + z \end{matrix} \right| i\beta x\right)
= \frac{1}{2\pi i} \int_{c-i\infty}^{c+i \infty} 
\Gamma(1 + z - w) \, \frac{(i\beta x)^w}{w^2} \, dw .
\end{equation}
In this formula, $c$ is chosen so that $c < 0$ and $c > z - 1/2$,  
ensuring that all poles of $\Gamma(-w)$ and $\Gamma(1+z-w)$ 
lie to the right of the path of integration,  
and that the condition  
\[
(q-p)\left(\Re(w) + \frac{1}{2}\right) >
\Re\!\left(\sum_{j=1}^q b_j - \sum_{i=1}^p a_i \right) + 1
\]
is satisfied for $p = 2$, $q = 3$, $a_1 = a_2 = 1$, $b_1 = b_2 = 0$, and $b_3 = 1+z$,  
since $|\arg(i\beta x)| = \delta \pi$ with $\delta = m + n - (p+q)/2$ for $m = 3$ and $n = 0$  
(see \cite[\S5.2]{Lu69}). 

For $w = i v$ ($ v \in \R$), we have  
$|(i\beta x)^w| = \exp(-(\pi/2)v)$. Therefore, using the Stirling formula (\cite[(5.11.9)]{NIST}), 
\[
\Gamma(1 + z - w) \, \frac{(i\beta x)^w}{w^2} 
\ll_{z} 
|v|^{z+1/2} \exp(-(\pi/2)|v|) \cdot \frac{\exp(-(\pi/2)v)}{|v|^2}
\ll |v|^{z-3/2}
\]
for fixed \( z \geq -1 \). 
Hence, if $-1 \leq z < 1/2$, 
the right-hand side of \eqref{EQ_306} converges absolutely and is bounded 
in $x$, where the bound depending on $z$.  

In the case $z = 1/2$, letting $c \to 0^-$ and noting that  
\[
\Gamma(s) = \sqrt{\frac{2\pi}{s}}\left(\frac{s}{e}\right)^s
\left( 1 + \frac{1}{12s} + \frac{1}{288s^2} + \dots \right)
\]
(\cite[(5.11.4), (5.11.10)]{NIST}),  
it suffices to show that
\begin{equation} \label{EQ_307}
\int_{1}^{\infty} 
\frac{1}{\sqrt{3/2+iv}}\left(\frac{3/2+iv}{e}\right)^{3/2+iv}\, 
\frac{(i\beta x)^{-iv}}{v^2} \, dv 
\end{equation}
is bounded in $x$.
This equals 
\[
\begin{aligned}
& e^{-3/2} \int_{1}^{\infty} \frac{|3/2+iv|}{v^2}
\exp\bigl( v(\tfrac{\pi}{2} - \arg(3/2+iv)) \bigr) \\
&\qquad \times 
\exp\bigl( i(v \log|3/2+iv| + \arg(3/2+iv)) \bigr) \, (e \beta x)^{-iv} \, dv .
\end{aligned}
\]
Since
\[
\arg\!\left( \frac{3}{2} + i v \right) = \frac{\pi}{2} - \frac{3}{2v} + O(v^{-3}), \quad v>1, 
\]
we get
\[
\eqref{EQ_307}= i \int_{1}^{\infty} \frac{1}{v}
\exp\bigl( i v (\log v - \log (e\beta x)) \bigr) \, dv + O(1).
\]
Applying Lemma~\ref{lem_1} to  
$f(v) = v \bigl(\log v - \log (e\beta x)\bigr)$,  
$g(v) = 1/v$, and the interval $[V, 2V]$ with $V > \beta x$  
shows that the above integral converges.  

We now prove that it is bounded in $x$.
For $v > 1$, the equation \( f'(v) = \log (v / (\beta x)) = 0 \) has the unique solution \( v = \beta x \).
Let $v_0 \,(< \beta x)$ satisfy $g(v)/f'(v) = -1$ and 
$v_1 \,(> \beta x)$ satisfy $g(v)/f'(v) = 1$.
Then \( -1 \leq g(v)/f'(v) < 0 \) on \([1, v_0]\) and \( 0 < g(v)/f'(v) \leq 1 \) on \([v_1, \infty)\).
Moreover, \( (g(v)/f'(v))' = 0 \) only at \( v = \beta x / e \), and \( g(v)/f'(v) \) is monotone on each of
\([1, \beta x / e]\), \( [\beta x / e, v_0] \), and \([v_1, \infty)\).
Hence, by Lemma~\ref{lem_1}, the integrals of \( g(v) e^{i f(v)} \) 
over these intervals are bounded in $x$. 
Finally, for the remaining part we have
\[
\left| \int_{v_0}^{v_1} g(v) e^{if(v)} \, dv \right| 
\leq  \int_{v_0}^{v_1} | g(v) | \, dv 
\leq \log v_1 - \log v_0 .
\]
By the definitions of $v_0$ and $v_1$, the right-hand side equals $v_1^{-1} + v_0^{-1}$, and since both $v_0$ and $v_1$ increase as $x$ increases, $v_1^{-1} + v_0^{-1}$ is bounded with respect to $x$.
\end{proof}

\begin{lemma} \label{lem_3} 
Let $x>1$ and $\alpha>0$. 
Then we have
\begin{equation} \label{EQ_308}
\aligned 
G(x; \alpha,iy) 
& = (2\pi i\alpha)^{-1-iy}\,\Gamma(1+iy)   \\
& \qquad \times 
 \left( 1 + \frac{1}{\log x}\left( \log (2\pi i \alpha) - \frac{\Gamma'}{\Gamma}(1+iy) \right) \right) 
\\
& \quad  -\frac{e^{-2\pi i \alpha}}{1+iy}{}_1F_1(1;2+iy;  2\pi i \alpha) \\
& \quad 
+ O\left(\frac{1}{\log x}\,|y|^{-2}\right)
+ O\left(\frac{x^{-\delta}}{\log x}\,|y|^{-(1/2-\delta)}\right) \\
\endaligned 
\end{equation}
as $|y| \to \infty$ for real $y$ and for any $0<\delta<1/2$, 
where ${}_1F_1(a; c; z)$ denotes the Kummer confluent hypergeometric function, 
and the implied constants may depend on $\delta$.
The first term on the right-hand side of \eqref{EQ_308} admits the following asymptotic formula
\begin{equation} \label{EQ_309}
\aligned 
(2\pi i\alpha)^{-1-iy}
& \,\Gamma(1+iy)  
 \left( 1 + \frac{1}{\log x}\left( \log (2\pi i \alpha) - \frac{\Gamma'}{\Gamma}(1+iy) \right) \right) \\
& = \frac{1}{2\pi i \alpha}\,e^{\pi i/4} \sqrt{2\pi|y|} \, 
e\left(
\frac{|y|}{2\pi}\log \frac{|y|}{2\pi \alpha e} 
\right) \\
& \qquad \times 
 \left( 1 - \frac{1}{\log x}\log\frac{|y|}{2\pi \alpha} \right) 
\left( 
1 +O(|y|^{-1})
\right) 
\exp\left(
-\frac{\pi}{2}(|y|-y)
\right)
\endaligned 
\end{equation}
as $|y| \to \infty$, where the implied constant is absolute.
\end{lemma} 
\begin{proof} Let $z=iy$. 
By combining \cite[(8.5.1), (8.19.1)]{NIST} with $\Gamma(s,z)=\Gamma(s)-\gamma(s,z)$, we have
\[
E_{-z}(2\pi i\alpha) 
= (2\pi i\alpha) ^{-1-z} \Gamma(1+z)
 -\frac{e^{-2\pi i \alpha}}{1+z}{}_1F_1(1,2+z,  2\pi i \alpha).
\]
Substituting this into the first term on the right-hand side of \eqref{EQ_305}, we obtain
\[
\aligned 
G(x;\alpha,z)
& =(2\pi i\alpha)^{-1-z}\Gamma(1+z)   \\
& \qquad \times 
 \left( 1 + \frac{1}{\log x}\left( \log (2\pi i \alpha) - \frac{\Gamma'}{\Gamma}(1+z) \right) \right) \\
& \quad -\frac{e^{-2\pi i \alpha}}{1+z}{}_1F_1(1,2+z,  2\pi i \alpha) \\
& \quad  + \frac{1}{\log x} (2\pi i\alpha)^{-1-z} 
G_{2,3}^{3,0}\left(\left.\begin{matrix} 1,\,1 \\ 0,\,0,\,1+z \end{matrix} \right| 2 \pi i\alpha x \right) 
\\
& \quad 
-  \frac{1}{\log x} (2\pi i\alpha)^{-1-z} \left[
G_{2,3}^{3,0}\left(\left.\begin{matrix} 1,\,1 \\ 0,\,0,\,1+z \end{matrix} \right| 2 \pi i\alpha \right)
\right. 
 \\
& \qquad \qquad \qquad \qquad \qquad + 
\left.
\Gamma(1+z)   
  \left( \log (2\pi i \alpha) - \frac{\Gamma'}{\Gamma}(1+z) \right)
\right].
\endaligned 
\]
We evaluate the latter terms on the right-hand side that contain $(\log x)^{-1}$ as a factor. 
By \cite[(16.19.2)]{NIST}, we have
\[
(2\pi i\alpha)^{-1-z}  
G_{2,3}^{3,0}\left(\left.\begin{matrix} 1,\,1 \\ 0,\,0,\,1+z \end{matrix} \,\right| 2 \pi i\alpha x \right) 
= 
x^{1+z}\,
G_{2,3}^{3,0}\left(\left.\begin{matrix} -z,\,-z \\ -1-z,\,-1-z,\,0 \end{matrix} \,\right| 2 \pi i\alpha x \right) .
\]
On the right-hand side, we have
\[
x^{1+z}\,
G_{2,3}^{3,0}\left(\left.\begin{matrix} -z,\,-z \\ -1-z,\,-1-z,\,0 \end{matrix} \,\right| 2 \pi i\alpha x \right)
=
\frac{1}{2\pi i}
\int_{c-i\infty}^{c+i\infty} \frac{\Gamma(-w)(2\pi i \alpha)^w x^{1+z+w}}{(1+z+w)^2} \, dw
\]
from \cite[\S5.2]{Lu69} again.   
Here the contour is chosen so that the poles of $\Gamma(-1-z-w)$ and $\Gamma(-w)$ 
lie to the right of the contour, that is $c<-1$. Therefore,
\begin{equation} \label{EQ_310}
\aligned 
(2\pi i\alpha)^{-1-z} 
& G_{2,3}^{3,0}\left(\left.\begin{matrix} 1,\,1 \\ 0,\,0,\,1 +z \end{matrix} \,\right| 2 \pi i\alpha x\right) \\
& = \frac{1}{2\pi i}
\int_{c-i\infty}^{c+i\infty} \frac{\Gamma(-w)(2\pi i \alpha)^w x^{1+z+w}}{(1+z+w)^2} \, dw.
\endaligned 
\end{equation}
When $x=1$, shifting the contour to $\Re(w)=c'$ with $-1 < c' <0$ 
takes the residue at $w=-1-z$, giving
\[
\aligned  
\, & 
(2\pi i\alpha)^{-1-z} 
G_{2,3}^{3,0}\left(\left.\begin{matrix} 1,\,1 \\ 0,\,0,\,1 +z \end{matrix} \,\right| 2 \pi i\alpha \right) \\
& \qquad +(2\pi i\alpha)^{-1-z}\Gamma(1+z)   
  \left( \log (2\pi i \alpha) - \frac{\Gamma'}{\Gamma}(1+z) \right) 
\\
& =  \frac{1}{2\pi i}
\int_{c'-i\infty}^{c'+i\infty} \frac{\Gamma(-w)(2\pi i \alpha)^w x^{1+z+w}}{(1+z+w)^2} \, dw.
\endaligned 
\]
Shifting the contour further to the right (which is justified by the Stirling formula),  
we first take the residue at $w=0$, then at $w=n$ ($n \in \mathbb{Z}_{>0}$) successively.  
Thus,
\[
\aligned
\, & 
(2\pi i\alpha)^{-1-z} 
G_{2,3}^{3,0}\left(\left.\begin{matrix} 1,\,1 \\ 0,\,0,\,1 +z \end{matrix} \,\right| 2 \pi i\alpha \right) \\
& \qquad +(2\pi i\alpha)^{-1-z}\Gamma(1+z)   
  \left( \log (2\pi i \alpha) - \frac{\Gamma'}{\Gamma}(1+z) \right) 
 \ll  \frac{1}{y^2}. 
\endaligned 
\]
It remains to bound
\[
(2\pi i\alpha)^{-1-z}  
G_{2,3}^{3,0}\left(\left.\begin{matrix} 1,\,1 \\ 0,\,0,\,1+z \end{matrix} \,\right| 2 \pi i\alpha x \right).
\]
By \cite[(16.19.2)]{NIST}, taking $c=-1-\delta$ in the integral representation \eqref{EQ_310}, 
the Stirling formula gives
\[
\aligned 
\, & (2\pi i\alpha)^{-1-z}  
G_{2,3}^{3,0}\left(\left.\begin{matrix} 1,\,1 \\ 0,\,0,\,1+z \end{matrix} \,\right| 2 \pi i\alpha x \right) \\
& \quad =
\frac{1}{2\pi}
\int_{-\infty}^{\infty} \frac{\Gamma(1+\delta-iv) 
(2\pi i \alpha)^{-1-\delta+iv} x^{-\delta+i(y+v)}}{(-\delta+i(y+v))^2} \, dv \\
& \quad \ll x^{-\delta}
\int_{-\infty}^{\infty} 
\frac{|1+\delta-iv|^{\delta+1/2}}{|\delta-i(y+v)|^2} \, dv 
\ll_\delta x^{-\delta} y^{-1/2+\delta}.
\endaligned 
\]
Combining the above estimates yields \eqref{EQ_308}.
\end{proof}

Since the estimate \eqref{EQ_308} for $G(x;\alpha,iy)$ 
is inadequate for the purposes of later applications when $y > 0$ is large, 
we introduce the following lemma as an alternative.

\begin{lemma} \label{lem_4}
Let $x>e^2$, $\alpha>0$, and $y>0$. Then we have 
\begin{equation} \label{EQ_311}
G(x;\alpha,iy)
\ll
\begin{cases}
\displaystyle{
\frac{x}{\log x} \cdot \frac{1}{y}} & \text{for}~ 
2\pi \alpha x <y\leq 2\pi \alpha x + \sqrt{2\pi \alpha x}, \\[10pt]
\displaystyle{
\frac{1}{y} + \frac{x}{\log x} \cdot \frac{1}{(y-2\pi\alpha x)^2} 
 } & \text{for}~ y > 2\pi \alpha x + \sqrt{2\pi \alpha x}, 
\end{cases}
\end{equation}
where the implied constant may depend on $\alpha$. 
\end{lemma} 
\begin{proof}
Let $y>0$, and set $\beta = 2\pi\alpha$ as before.  
Define $f(u) = -\beta x u + y \log u$ and $g(u) = \log u$, and set
\[
H(x;\beta,y) := \int_{1/x}^{1} g(u) e^{i f(u)} \, du.
\]
Then
\[
G(x; \alpha, i y) = - \frac{x^{1+iy}}{\log x} \, H(x; \beta, y).
\]

First, we evaluate the integral $H(x; \beta, y)$ for 
$\beta x < y \leq \beta x + \sqrt{\beta x}$.
Since $y > \beta x$, we have
$f'(u) = -\beta x + y/u \neq 0$ on the integration interval $[1/x, 1]$
(since $f'(u) = 0$ would imply $u = y / (\beta x) \leq 1$), 
and hence
\[
\frac{g(u)}{f'(u)} = \frac{u\log u}{\,y - \beta x u\,}
\]
is bounded there.  
We examine the monotonicity of this function. Its derivative is
\[
\left( \frac{g(u)}{f'(u)} \right)' 
= \frac{y - \beta x u + y \log u}{(y - \beta x u)^2}.
\]
If $y > \beta x$, then we have $f'(u) = -\beta x + y/u > 0$ for $0 < u < 1$, 
so the numerator in 
$(g(u)/f'(u))'$ vanishes for at most one point $u_0$, which satisfies
$\log u = -1 + (\beta x u)/y$.  
Thus $u_0 > 1/e$, and we have $1/x \leq u_0 \leq 1$ if $x > e$.  
Hence, if $x > e$ and $y > \beta x$, there exists exactly one point $u_0 \in [1/x,1]$ where 
$(g(u)/f'(u))'=0$.  
Dividing $[1/x,1]$ into $[1/x,u_0]$ and $[u_0,1]$, 
we see that $g(u)/f'(u)$ is monotone on each interval.

At $u_0$, we have $y - \beta x u_0 + y \log u_0 = 0$,
and thus
\[
\frac{g(u_0)}{f'(u_0)} 
= \frac{u_0\log u_0}{\,y - \beta x u_0\,}
= - \frac{u_0 \log u_0}{y \log u_0} 
= -\frac{u_0}{y}.
\]
Since $1/e < u_0 < 1$,
\[
\frac{1}{ey} \leq \left| \frac{g(u_0)}{f'(u_0)} \right| \leq \frac{1}{y}.
\]
At the endpoints $u=1/x$ and $u=1$, 
\[
\frac{g(1/x)}{f'(1/x)} = -\frac{\log x}{x} \cdot \frac{1}{y - \beta},
\quad
\frac{g(1)}{f'(1)} = 0.
\]
A direct check (using $x > 1$) shows that 
the maximal value of $|g(u)/f'(u)|$ in the interval $[1/x,1]$ is $u_0 / y \leq 1/y$.  
Applying Lemma~\ref{lem_1} to $f(u)$ and $g(u)$ on each interval 
yields the first part of the estimate in~\eqref{EQ_311}.

Second, we evaluate the integral $H(x; \beta, y)$ for 
$y > \beta x + \sqrt{\beta x}$. 
By $y > \beta x$, we have $g(1) = 0$ and $f'(1) = x (y - \beta) \neq 0$, hence
\[
\aligned
H(x; \beta, y) 
&= \int_{1/x}^1 \frac{g(u)}{i f'(u)} \, \frac{d}{du} e^{i f(u)} \, du \\
&= - \frac{g(1/x)}{i f'(1/x)} \, e^{i f(1/x)}
  - \int_{1/x}^1 \left( \frac{g(u)}{i f'(u)} \right)' e^{i f(u)} \, du \\
&= -\frac{\log x}{x^{1+iy}} \cdot \frac{e^{-i\beta}}{i (y - \beta)}
  - \frac{1}{i} \int_{1/x}^1 \left( \frac{g(u)}{f'(u)} \right)' e^{i f(u)} \, du.
\endaligned
\]

Let $G(u) = (g(u)/f'(u))'$. For the integrand on the right-hand side we have
\[
\left( \frac{G(u)}{f'(u)} \right)' 
= \frac{ (y - \beta x u)(2y + \beta x u) 
       + y (y + 2\beta x u) \log u }
       {(y - \beta x u)^4}.
\]
The numerator vanishes when
\[
\log u = \frac{-2 + (\beta x u / y) \bigl(1 + \beta x u / y \bigr)}
              {1 + 2(\beta x u / y)},
\]
and if $0 < \beta x u / y < 1$ this quantity lies strictly between $-2$ and $0$.  
Thus, if $u_1 \in [1/x,1]$ satisfies 
$(G(u)/f'(u))' = 0$, then $1/e^2 < u_1 < 1$.  
In particular, if $x > e^2$, then $1/x < u_1 < 1$, and 
$G(u)/f'(u)$ is monotone on $[1/x,u_1]$ and $[u_1,1]$. 

At the endpoints $u=1/x$ and $u=1$, 
\[
\frac{G(1/x)}{f'(1/x)}
= \frac{1}{x(y - \beta)^2} - \frac{\log x}{x} \cdot \frac{y}{(y - \beta)^3},
\quad
\frac{G(1)}{f'(1)} = \frac{1}{(y - \beta x)^2}.
\]
At $u_1$,
\[
\frac{G(u_1)}{f'(u_1)}
= \frac{u_1 \bigl( (y - \beta x u_1) + y \log u_1 \bigr)}
       {(y - \beta x u_1)^3}.
\]
From $(y - \beta x u_1)(2y + \beta x u_1) 
+ y (y + 2\beta x u_1) \log u_1 = 0$ we deduce
\[
y - \beta x u_1 
= - \frac{y (y + 2\beta x u_1) \log u_1}{\,2y + \beta x u_1\,},
\]
and substituting this for one occurrence of $y - \beta x u_1$ 
in both the numerator and the denominator 
of $G(u_1)/f'(u_1)$ 
yields
\[
\frac{G(u_1)}{f'(u_1)}
= - \frac{u_1}{(y - \beta x u_1)(y + 2\beta x u_1)}.
\]
Since $0 < 1/x < u_1 < 1$ and $y - \beta x > 0$,
\[
\left| \frac{G(u_1)}{f'(u_1)} \right|
\leq \frac{1}{(y - \beta x u_1)(y + 2\beta x u_1)}
\leq \frac{1}{(y - \beta x)^2}
= \frac{G(1)}{f'(1)}.
\]
Moreover, for $x > e^2$,
\[
\left| \frac{G(1/x)}{f'(1/x)} \right|
= \frac{1}{(y - \beta)^2} \cdot \frac{1}{x}
  \left( \frac{y \log x}{y - \beta} - 1 \right)
< \frac{\log x}{x (y - \beta)^2}
< \frac{1}{(y - \beta x)^2}
= \frac{G(1)}{f'(1)}.
\]
Applying Lemma~\ref{lem_1} to $f(u)$ and $G(u)$ on the two intervals 
$[1/x,u_1]$ and $[u_1,1]$ yields the second part of the estimate in~\eqref{EQ_311}. 
\end{proof} 

\section{Proofs of Theorems \ref{thm_1} and \ref{thm_0724}} \label{section_4}

\subsection{Proof of Theorem \ref{thm_1}}

We note that  establishing the following proposition 
will imply Theorem~\ref{thm_1}.

\begin{proposition}[explicit formula] \label{prop_3}
Let $t \in \R$ and $\alpha>0$.  
Assuming RH, we have
\begin{equation} \label{EQ_401}
\aligned 
\Psi(x;1/2+it,\alpha) 
& = - e^{-\pi i/4} \sqrt{2\pi} \, Z(x;t,\alpha)  +A_5(t,\alpha)
+O\left(\frac{1}{\log x}\right)
\endaligned 
\end{equation}
as $x \to \infty$,  
where $A_5(t, \alpha)$ is a constant depending on $t$ and $\alpha$,  
and the implied constant may also depend on $t$ and $\alpha$.
\end{proposition}

By Proposition~\ref{prop_1}, the GRH for all $\chi \bmod q$  
is equivalent to the statement that  
the left-hand side of \eqref{EQ_401} satisfies the appropriate bound  
for all $\alpha = a/q_1$ with $1 \leq a \leq q_1$, $(a,q_1) = 1$, and $q_1 \mid q$.  
Hence, Propositions~\ref{prop_1} and \ref{prop_3} together yield  
\eqref{EQ_104}, which is the first part of Theorem~\ref{thm_1}.  
Furthermore, combining \eqref{EQ_0215} in Proposition~\ref{prop_2}  
with \eqref{EQ_401} in Proposition~\ref{prop_3}  
gives \eqref{EQ_105}, which is the second part of Theorem~\ref{thm_1}.  
Therefore, proving Proposition~\ref{prop_3}  
completes the proof of Theorem~\ref{thm_1}. 
\hfill $\Box$

\subsection{Auxiliary lemma} 

\begin{lemma} \label{lem_5}
Let $\alpha > 0$, and let $s$ be a complex number with $\Re(s) > 0$, $s \neq 1$, and not coinciding with any nontrivial zero of $\zeta(s)$. Then 
\begin{equation} \label{EQ_402}
\aligned 
\Psi(x; s, \alpha)
& = \frac{2\pi i \alpha}{1-s}\, G(x; \alpha, 1-s) \\ 
& \quad 
- 2\pi i \alpha \left[ \sum_{\rho} \frac{1}{\rho - s} \, G(x; \alpha, \rho - s) 
+ \frac{\zeta'}{\zeta}(s) \, G(x; \alpha, 0) \right] \\[2mm]
& \quad 
+ A_6(s, \alpha) + O\!\left(\frac{1}{\log x}\right)
\endaligned 
\end{equation}
as $x \to \infty$,  
where $A_6(s, \alpha)$ is a constant depending on $s$ and $\alpha$,  
and the implied constant may also depend on $s$ and $\alpha$.
\end{lemma}
\begin{proof}

By integration by parts, we have
\[
\aligned 
- 2 \pi i \alpha \ 
\int_{n}^{x} e(-\alpha u) \left( 1 - \frac{\log u}{\log x} \right) 
 \, du
& = -
e(-\alpha n) \left( 1 - \frac{\log n}{\log x} \right) 
+ \frac{1}{\log x}
\int_{n}^{x} \frac{e(-\alpha u)}{u} \, du. 
\endaligned 
\]
Using the integral formula of $E_s(z)$ for $s=1$ in \cite[(6.2.1), (8.19.2)]{NIST}, 
we obtain
\[
\int_{n}^{x} \frac{e(-\alpha u)}{u}  \, du 
= - E_1(2\pi i\alpha x) + E_1(2\pi i\alpha n) . 
\]
Therefore,
\[
\aligned 
e(-\alpha n) \left( 1 - \frac{\log n}{\log x} \right)  
& =2\pi i \alpha
\int_{n}^{x} e(-\alpha u) \left( 1 - \frac{\log u}{\log x} \right) 
 \, du \\
& \quad 
-  \frac{1}{\log x}(E_1(2\pi i \alpha x)-E_1(2\pi i \alpha n)).
\endaligned 
\]
Substituting this into \eqref{EQ_109} and 
interchanging the order of summation and integration, 
we obtain
\[
\aligned 
\Psi(x;s,\alpha) 
& = 2\pi i \alpha\int_{1}^{x} e(-\alpha u) \left( 1 - \frac{\log u}{\log x} \right)
 \sum_{n \leq u} \frac{\Lambda(n)}{n^s}
 \, du \\
& \quad 
- \frac{1}{\log x} \sum_{n \leq x} \frac{\Lambda(n)}{n^s}
(E_1(2\pi i\alpha x)-E_1(2\pi i\alpha n)). 
\endaligned 
\]
For the second term on the right-hand side, recall that for $|\arg z| < 3\pi/2$,
\[
E_1(z) = \frac{e^{-z}}{z}\big(1+O(|z|^{-1})\big)
\quad \text{as} \quad |z| \to \infty
\]
(\cite[(2.11.6)]{NIST}). 
Hence, 
\[
\aligned 
\sum_{n \leq x} \frac{\Lambda(n)}{n^s}(E_1(2\pi i\alpha x)-E_1(2\pi i\alpha n))
& \ll_\alpha \sum_{n \leq x} \frac{\Lambda(n)}{n^\sigma}\left(\frac{1}{x}+\frac{1}{n} \right) 
 \ll \sum_{n \leq x} \frac{\Lambda(n)}{n^{1+\sigma}} \ll_\sigma 1,
\endaligned 
\]
where $\sigma=\Re(s)$. 
Therefore,
\begin{equation} \label{EQ_403}
\Psi(x;s,\alpha) 
 = 2\pi i \alpha \int_{1}^{x} e(-\alpha u) \left( 1 - \frac{\log u}{\log x} \right)
 \sum_{n \leq u} \frac{\Lambda(n)}{n^s}
 \, du
+O\left(\frac{1}{\log x}\right), 
\end{equation} 
where the implied constant depends on $s$. 

To evaluate the integral on the right-hand side of \eqref{EQ_403},  
we recall the explicit formula
\begin{equation} \label{EQ_404}
\sideset{}{'}\sum_{n \leq x} \frac{\Lambda(n)}{n^s} 
= 
 \frac{x^{1-s}}{1-s} 
 - \sum_{\rho} \frac{x^{\rho-s}}{\rho-s} 
-\frac{\zeta'(s)}{\zeta(s)}
+ \frac{1}{2}u^{-s-2} \Phi(u^{-2},1,1+s/2)
\end{equation} 
for $x>1$ and $s \neq 1, \rho$~(\cite[Exercises 12.1.1, 4]{MV07}),  
where $\sideset{}{'}\sum_{n \leq x}a_n$ denotes  
$\sum_{n \leq x}a_n - \frac12 a_x$ when $x$ is a prime power,  
and $\sum_{n \leq x}a_n$ otherwise,  
the sum $\sum_\rho$ is taken over all nontrivial zeros of $\zeta(s)$ counted with multiplicity,  
and is defined as $\sum_\rho = \lim_{T \to \infty}\sum_{|\Im(\rho)| \leq T}$, 
and  $\Phi(z,s,\alpha)=\sum_{n=0}^{\infty}(n+\alpha)^{-s}z^n$  
is the Lerch transcendent.  
Using \eqref{EQ_404}, and proceeding formally, we obtain
\begin{equation} \label{EQ_405}
\aligned 
\int_{1}^{x}  & e(-\alpha u) \left( 1 - \frac{\log u}{\log x} \right)
 \sum_{n \leq u} \frac{\Lambda(n)}{n^s} \, du \\
& = - \int_{1}^{x}  e(-\alpha u)
\sum_{\rho} \frac{u^{\rho-s}}{\rho-s} 
\left( 1 - \frac{\log u}{\log x} \right)
\, du \\
& \quad 
+ \int_{1}^{x}  e(-\alpha u)
\left( \frac{u^{1-s}}{1-s} 
-\frac{\zeta'}{\zeta}(s) 
\right)
\left( 1 - \frac{\log u}{\log x} \right)
\, du \\
& \quad 
+ \frac{1}{2}\int_{1}^{x} e(-\alpha u)\,
u^{-s-2} \Phi(u^{-2},1,1+s/2) \left( 1 - \frac{\log u}{\log x} \right)
\,  du \\
& = \frac{1}{1-s}\,G(x; \alpha,1-s) \\
& \quad - \sum_{\rho} \frac{1}{\rho-s} G(x; \alpha,\rho-s) 
-\frac{\zeta'}{\zeta}(s)\,G(x; \alpha,0) \\
& \quad 
+\frac{1}{2}\int_{1}^{x}  e(-\alpha u)  
u^{-s-2} \Phi(u^{-2},1,1+s/2)\left( 1 - \frac{\log u}{\log x} \right)
\,  du.
\endaligned 
\end{equation} 
As noted after \cite[Theorem~12.5]{MV07}, 
the sum $\sum_{\rho} x^{\rho}/\rho$ 
converges uniformly for $x$ in any interval of the form $[n_1+\delta,\, n_2-\delta]$,  
where $n_1 < n_2$ are consecutive prime powers,  
and converges boundedly in a neighborhood of a prime power. 
From these facts, the sum $\sum_{\rho} x^{\rho-s}/(\rho-s)$  
in \eqref{EQ_404} has the same convergence property, because 
\[
\sum_{\rho} \frac{x^{\rho-s}}{\rho-s}
= 
x^{-s}\sum_{\rho} \frac{x^\rho}{\rho}
+ s \sum_{\rho}  \frac{x^{\rho-s}}{\rho(\rho-s)}
\]
and the second sum on the right-hand side converges absolutely and uniformly in $x$  
on any compact set. 
Therefore, the interchange of the order of integration and summation  
above is justified.

Furthermore, the integral on the right-hand side of \eqref{EQ_405} can be evaluated as follows
\begin{equation} \label{EQ_406}
\aligned 
 \int_{1}^{x}  & e(-\alpha u) 
u^{-s-2} \Phi(u^{-2},1,1+s/2)\left( 1 - \frac{\log u}{\log x} \right)
\,  du \\
& = 
\frac{1}{\log x}\int_{1}^{x} \frac{1}{t} 
\left[ \int_{1}^{t}  e(-\alpha u) 
u^{-s-2} \Phi(u^{-2},1,1+s/2)
\,  du \right] dt \\
& = 
 \frac{1}{\log x}\int_{1}^{x} \frac{1}{t} 
\left[ \int_{1}^{\infty}  e(-\alpha u) \,
u^{-s-2} \Phi(u^{-2},1,1+s/2)
\,  du + O(t^{-1-\sigma}) \right] dt \\
& = 
\int_{1}^{\infty} e(-\alpha u) 
u^{-s-2} \Phi(u^{-2},1,1+s/2)
\,  du + O\left(\frac{1}{\log x}\right), 
\endaligned 
\end{equation}
since $\Phi(u^{-2},1,1+s/2) \to 2/(2+s)$ as $u \to \infty$ by definition.  
Combining \eqref{EQ_403}, \eqref{EQ_405}, and \eqref{EQ_406},  
we obtain \eqref{EQ_402}.
\end{proof}

\subsection{Proof of Proposition~\ref{prop_3} }  

Assuming RH, writing the nontrivial zeros of $\zeta(s)$ 
as $\rho = 1/2 + i\gamma$, 
and setting $s = 1/2 + it$ with $t \in \mathbb{R}$ and $t \neq \gamma$, 
\[
\sum_\rho
\frac{1}{\rho-s} G(x; \alpha,\rho-s)
= 
\frac{1}{i}
\sum_\gamma 
\frac{1}{\gamma-t} G(x; \alpha,i(\gamma-t)). 
\] 
We divide the sum on the right-hand side into two parts.

First, consider the sum over $\gamma$ with $\gamma - t > 2\pi \alpha x$.  
For convenience, set $X = 2\pi \alpha x$.  
From the first bound in~\eqref{EQ_311} 
and the asymptotic formula $N(T)=(T/\pi)\log(T/(2\pi e))+O(\log T)$ 
for the number of nontrivial zeros of $\zeta(s)$ 
with $|\Im(s)| \leq T$ (\cite[Corollary 14.3]{MV07}), we obtain
\[
\aligned 
\sum_{X \leq \gamma -t  \leq X+\sqrt{X}} 
& \frac{1}{\gamma-t}G(x; \alpha,i(\gamma-t)) \\
& \ll_\alpha \frac{x}{\log x} 
\sum_{X \leq \gamma -t  \leq X+\sqrt{X}} 
\frac{1}{(\gamma-t)^2} \\
& \ll \frac{x}{\log x} \int_{X}^{X+\sqrt{X}} \frac{dN(y)}{y^2} 
\ll \frac{x}{\log x} \int_{X}^{X+\sqrt{X}} \frac{\log y}{y^2} \, dy \\
& \ll \frac{x}{\log x} \frac{\log X}{X\sqrt{X}}(1+O(X^{-1/2}))\ll \frac{1}{\sqrt{x}}. 
\endaligned 
\]
On the other hand, from the second bound in~\eqref{EQ_311}, we have
\[
\aligned 
\sum_{\gamma-t \geq X+\sqrt{X}} 
& \frac{1}{\gamma-t}G(x; \alpha,i(\gamma-t)) \\
& \ll_\alpha 
\sum_{\gamma-t \geq X+\sqrt{X}} \frac{1}{(\gamma-t)^2} +
\frac{x}{\log x}\sum_{\gamma-t \geq X+\sqrt{X}} \frac{1}{(\gamma-t)(\gamma-t-2\pi\alpha x)^2} 
 \\
& \ll 
\int_{X+\sqrt{X}}^\infty \frac{\log y}{y^2} \, dy\frac{x}{\log x} 
+ 
\int_{X+\sqrt{X}}^\infty \frac{\log y}{y(y-2\pi\alpha x)^2} \, dy \\
& =  O\left(\frac{\log x}{x}\right) + O\left(\frac{1}{\sqrt{x}}\right). 
\endaligned 
\]
Hence
\begin{equation} \label{EQ_407}
\sum_{\gamma-t > 2\pi \alpha x} \frac{1}{\gamma-t}G(x; \alpha,i(\gamma-t))
= O\left(\frac{1}{\sqrt{x}}\right). 
\end{equation}

Next, consider the sum over $\gamma$ with $\gamma - t \leq 2\pi \alpha x$. 
Here $\gamma - t \neq 0$ by assumption,  
and the convergence in~\eqref{EQ_212} allows us to proceed as follows.  
The first term on the right-hand side of~\eqref{EQ_308} decays rapidly 
for $y = \gamma - t < 0$ by~\eqref{EQ_309};  
for $y = \gamma - t > 0$, the contribution of the error term in~\eqref{EQ_309} is $O(|y|^{-1/2})$.  
Thus
\[
\aligned 
\sum_{\gamma-t \leq 2\pi \alpha x} 
& \frac{1}{\gamma-t}
(2\pi i\alpha)^{-1-i(\gamma-t)}\Gamma(1+i(\gamma-t)) \\
& \times \left( 1 + \frac{1}{\log x}\left( \log (2\pi i \alpha) - \frac{\Gamma'}{\Gamma}(1+i(\gamma-t)) \right) \right) 
\\
& =\frac{1}{2\pi i \alpha}\,e^{\pi i/4} \sqrt{2\pi}\,Z(x;t,\alpha) 
+ A_7(t,\alpha) + O\left( \frac{1}{\log x} \right)
\endaligned 
\]
for some constant $A_7(t,\alpha)$.  
Moreover, from the series representation of the confluent hypergeometric function ${}_1F_1(1; s; z)$,  
we have 
\[
{}_1F_1(1;2 + iy; 2\pi i \alpha) \ll_\alpha 1
\]
uniformly for $|y| \geq 1$. Hence
\[
\sum_{\gamma-t<2\pi \alpha x} \frac{1}{\gamma-t}\cdot
\frac{e^{-2\pi i\alpha}}{1+i(\gamma-t)} {}_1F_1(1;2 + i(\gamma-t); 2\pi i\alpha) 
= A_8(t;\alpha) + O\left( \frac{1}{\log x} \right) 
\]
for some constant $A_8(t,\alpha)$.  
The contribution of the error term in~\eqref{EQ_308} is also bounded by~\eqref{EQ_212},  
and therefore
\begin{equation} \label{EQ_408}
\aligned 
\sum_{\gamma-t<2\pi \alpha x} & \frac{1}{\gamma-t}\,G(x;\alpha,i(\gamma-t)) \\
& = \frac{1}{2\pi i \alpha}\,e^{\pi i/4} \sqrt{2\pi}\,Z(x;t,\alpha) 
+ A_9(t,\alpha) + O\left( \frac{1}{\log x} \right)
\endaligned 
\end{equation}
for some constant $A_9(t,\alpha)$. 

Applying the above estimates~\eqref{EQ_407} and~\eqref{EQ_408},  
together with~\eqref{EQ_301}, to~\eqref{EQ_402},  
we obtain~\eqref{EQ_401} of Proposition~\ref{prop_3}.
\hfill $\Box$

\subsection{Proof of Theorem \ref{thm_0724}}

Let $\chi \bmod q$ be a Dirichlet character  satisfying the assumptions of the theorem 
and let $\alpha = a/q$.  
Multiplying both sides of \eqref{EQ_401} by $\overline{\chi}(a)$ and summing over $1 \leq a \leq q$,  
we obtain
\begin{equation} \label{EQ_409}
\aligned 
\sum_{a=1}^q \bar{\chi}(a)
\Psi(x;1/2+it,a/q) 
& = - e^{-\pi i/4} \sqrt{2\pi} \sum_{a=1}^q \bar{\chi}(a) Z(x;t,a/q)  
+O\left(\frac{1}{\log x}\right)
\endaligned 
\end{equation}
by the orthogonality of Dirichlet characters. 
The left-hand side here is $\Psi(x;1/2+it,\chi)$ by \eqref{EQ_206},  
and hence, by the integral representation \eqref{EQ_207},  
the validity of \eqref{EQ_106} is equivalent to the GRH for $L(s,\chi)$.

Furthermore,  
multiplying both sides of \eqref{EQ_0215} by $\overline{\chi}(a)$ and summing over $1 \leq a \leq q$,  
we see that the left-hand side coincides with that of \eqref{EQ_409}.  
Comparing the right-hand sides with \eqref{EQ_409}, 
we obtain \eqref{EQ_107} and \eqref{EQ_108}. 
Note that $\tau(\chi_0)=\mu(q)$ and $m(t,\chi_0)=0$ by the assumption $t \not=\gamma$. 
\hfill $\Box$

\subsection{Comparison with Fujii's result} \label{subsection_Fujii}

The explicit formula \eqref{EQ_401} is analogous to the one obtained by Fujii under RH. 
He proved that 
\begin{equation} \label{EQ_410}
\aligned 
\sum_{n \leq x} \Lambda(n)e(-\alpha n)
& = - e^{-\pi i/4} \frac{1}{\sqrt{\alpha}} 
\sum_{0<\gamma \leq 2\pi \alpha x} 
e\left(\frac{\gamma}{2\pi} \log \frac{\gamma}{2\pi e \alpha}\right) \\
& \quad + O\left(
\sqrt{x}\left(
\frac{\log x}{\log\log x}
\right)^2
\right)
\endaligned 
\end{equation}
as $x \to \infty$ for any fixed positive $\alpha$.  
(This was first announced in~\cite[Corollary~1]{Fu92}; however, compared with the theorem in the same paper, it contains a misprint, which is corrected in the fourth line of p.~229 in~\cite{Fu99}.)  

Using the explicit formula \eqref{EQ_410} together with the orthogonality of Dirichlet characters, and arguing in the same manner as in the proof of Theorem~\ref{thm_0724}, one finds that if $\tau(\chi) \ne 0$ (this condition is missing in Fujii's statement), 
then the GRH for $L(s,\chi)$ is equivalent to  
\[
\sum_{a=1}^q
\frac{\overline{\chi(a)}}{\sqrt{a/q}} 
\sum_{0<\gamma \leq 2\pi (a/q) x} 
e\left(\frac{\gamma}{2\pi} \log \frac{\gamma}{2\pi (a/q) e}\right) 
= O\left( x^{1/2+\varepsilon} \right) 
\]
as $x \to \infty$ for any $\varepsilon>0$ 
(cf. \cite[Proof of Theorem 2]{Fu88b}).  
However, unlike the latter part of Theorem~\ref{thm_0724}, one cannot detect zeros of $L(s,\chi)$ via the nontrivial zeros of $\zeta(s)$ in this setting.

\section{Proof of Theorem \ref{thm_3} } \label{section_5}

First, by partial summation, we have  
\[
\aligned 
\Psi(x; s, \alpha)
& = 
\sum_{n \leq x} \Lambda(n)e(-\alpha n) \frac{1}{n^s}
- \frac{1}{\log x}
\sum_{n \leq x} \Lambda(n)e(-\alpha n) \frac{\log n}{n^s}
\\
& = x^{-s}
\sum_{n \leq x} \Lambda(n)e(-\alpha n) 
+ s \int_{1}^{x} \left[ \sum_{n \leq u} \Lambda(n)e(-\alpha n) \right] u^{-s-1} \, du \\
& \quad 
- x^{-s}
\sum_{n \leq x} \Lambda(n)e(-\alpha n) 
+ \frac{1}{\log x} \int_{1}^{x} \left[ \sum_{n \leq u} \Lambda(n)e(-\alpha n) \right] 
u^{-s-1} \, du \\
& \quad 
- \frac{s}{\log x} \int_{1}^{x} \left[ \sum_{n \leq u} \Lambda(n)e(-\alpha n) \right] 
u^{-s-1}\log u \, du \\
& = 
\frac{1}{\log x} \int_{1}^{x} \left[ \sum_{n \leq u} \Lambda(n)e(-\alpha n) \right] 
u^{-s-1} \, du \\
& \quad 
+ \int_{1}^{x} \left[ \sum_{n \leq u} \Lambda(n)e(-\alpha n) \right] 
\left(
1 - \frac{\log u}{\log x}
\right)
u^{-s-1} \, du.
\endaligned 
\]
For the integrand on the right-hand side, 
if $|\alpha-a/q| \leq q^{-2}$ with $(a,q)=1$, then
\[
\sum_{n \leq x} \Lambda(n)e(-\alpha n) 
\ll \left( xq^{-1/2} + x^{4/5} + x^{1/2}q^{1/2} \right)(\log x)^4
\]
according to \cite[\S25, (2)]{Da80}. 
If $\alpha \in \R \setminus \Q$, this condition is satisfied for infinitely many rationals $a/q$.  
Choosing $q$ so that
\[
(\log x)^A \leq q \leq \frac{x}{(\log x)^A}, \quad A>8,
\]
we obtain
\[
\sum_{n \leq x} \Lambda(n)e(-\alpha n) 
\ll x (\log x)^{-A/2+4} = o(x).
\]
Using this bound, we find
\[
\aligned 
\int_{1}^{x} & \left[ \sum_{n \leq u} \Lambda(n)e(-\alpha n) \right] 
\left(
1 - \frac{\log u}{\log x}
\right)
u^{-s-1} \, du \\
& = \frac{1}{\log x} 
\int_{1}^{x} \frac{1}{t} 
\left( \int_{1}^{t} \left[ \sum_{n \leq u} \Lambda(n)e(-\alpha n) \right] 
u^{-s-1} \, du \right)  dt \\
& = o\left(\frac{1}{\log x} 
\int_{1}^{x} \frac{1}{t} 
 \int_{1}^{t} u^{-\sigma} \, du \,  dt \right) \\
& = o\left(\frac{1}{\log x} 
\left[ 
\frac{x^{1-\sigma}}{(1-\sigma)^2} + \frac{\log x}{\sigma-1} - \frac{1}{(1-\sigma)^2}
\right] \right)
 = o\left(\frac{x^{1-\sigma}}{\log x}  \right).
\endaligned 
\]
Hence, for $\alpha \in \R \setminus \Q$,
\begin{equation} \label{EQ_501}
\Psi(x;s,\alpha) = o\!\left(\frac{x^{1-\sigma}}{\log x}  \right).
\end{equation}

On the other hand, in the case $\alpha = a/q \in \Q$ with $(a,q)=1$, 
we use \eqref{EQ_205} for $q_1 = q$ together with 
the explicit formulas \eqref{EQ_209} and \eqref{EQ_211}.  
For the contribution of the nontrivial zeros in \eqref{EQ_211}, using the classical zero-free region 
$\beta > 1 - c/\log|\gamma|$, 
we have
\[
\sum_{\rho} \frac{x^{\rho-s}}{(\rho-s)^2}
\ll x^{1-\sigma} \sum_{\gamma>0} \frac{x^{-c/\log \gamma}}{\gamma^2}
\ll x^{1-\sigma} \int_{1}^{\infty} \frac{x^{-c/\log \gamma}}{\gamma^{1+\delta}} \, d\gamma.
\]
With the change of variables
\[
t = \sqrt{\frac{\delta}{c}} \, \frac{\log \gamma}{\sqrt{\log x}},
\]
and using the integral representation of the $K$–Bessel function
\[
K_{s}(z) = \frac{1}{2}\int_{0}^{\infty}
\exp\!\left(-\frac{z}{2}\left(t+\frac{1}{t}\right)\right) t^{s-1} \, dt,
\]
we obtain
\[
\aligned 
\int_{1}^{\infty}\frac{x^{-c/\log \gamma}}{\gamma^{1+\delta}} \, d\gamma
& = \sqrt{\frac{c}{\delta}} \sqrt{\log x}
\int_{0}^{\infty} \exp\!\left(
\sqrt{c\delta \log x}
\left(t+\frac{1}{t}\right)
\right) \,dt \\
& = 2\sqrt{\frac{c}{\delta}} \sqrt{\log x} \, K_1\!\left(2\sqrt{c\delta \log x}\right).
\endaligned 
\]
By the asymptotic formula
\[
K_s(z) = \sqrt{\frac{\pi}{2z}} \, e^{-z} \left(1+O(|z|^{-1})\right),
\]
this yields
\[
\sum_{\rho} \frac{x^{\rho-s}}{(\rho-s)^2}
\ll x^{1-\sigma} \exp(-c\sqrt{\log x})
\]
for some $c>0$.

Similarly, if $q \leq (\log x)^A$, applying the Siegel–Walfisz theorem  
(\cite[\S22]{Da80}) to \eqref{EQ_209}, 
the contribution of zeros can be handled in the same manner, giving
\begin{equation} \label{EQ_502}
\Psi(x;s,a/q) 
= \frac{x^{1-s}}{(1-s)^2\log x}\,(1+o(1)).
\end{equation}
Combining \eqref{EQ_501} and \eqref{EQ_502} with the explicit formula \eqref{EQ_401}, we obtain Theorem~\ref{thm_3}.
\hfill $\Box$

\section{Analogue of Theorem \ref{thm_0724}} \label{section_6} 

Fujii~\cite{Fu92} established a formula in which the sum on the right-hand side of the explicit formula~\eqref{EQ_410} is replaced by a sum over the nontrivial zeros of a Dirichlet $L$-function. 
In a similar manner, the sum on the right-hand side of~\eqref{EQ_401} 
can also be replaced by a sum over the nontrivial zeros of a Dirichlet $L$-function.  
Accordingly, for \eqref{EQ_105} in Theorem~\ref{thm_1} and 
\eqref{EQ_107} in Theorem~\ref{thm_0724}, one may likewise consider an analogue in which the right-hand side is replaced by the nontrivial zeros of $L(s,\psi)$ associated with a primitive Dirichlet character~$\psi$.

\begin{theorem} \label{thm_4}
Let $\psi$ be a primitive Dirichlet character modulo $k$ ($\geq 3$).  
Assume the GRH for $L(s,\psi)$, and set 
\[
Z_\psi(x;t,\alpha)
:= 
\sum_{0<\gamma_\psi - t \leq 2\pi \alpha x} 
\frac{1}{\sqrt{\gamma_\psi-t}}\,
e\!\left(\frac{\gamma_\psi-t}{2\pi}\log \frac{\gamma_\psi-t}{2\pi \alpha e} \right)  
\left(1 - \frac{1}{\log x}\log\frac{\gamma_\psi-t}{2\pi \alpha}\right), 
\]
where the sum is taken over all imaginary parts $\gamma_\psi$ 
of the nontrivial zeros $\rho_\psi = 1/2 + i\gamma_\psi$ of $L(s,\psi)$, counted with multiplicity. 
Let $t$ be a real number that is not equal to any $\gamma_\psi$.
\begin{enumerate}
\item Assuming  the GRH for all $L(s,\psi\chi)$ 
with Dirichlet characters $\chi \bmod q$, we have 
\begin{equation} \label{EQ_601}
\aligned  
Z_\psi(x;t,a/q_1) 
& = -\frac{e^{\pi i/4}\overline{\tau(\chi_1)}\,\chi_1(a)\delta(\chi_1)}{\sqrt{2\pi}\varphi(q_1)}\cdot 
 \frac{1}{(1/2 - it)^2} \cdot \frac{x^{1/2 - it}}{\log x}  \\
& \quad +\frac{e^{\pi i/4}}{2\sqrt{2\pi}\,\varphi(q_1)} 
\sum_{\chi \bmod q_1} \overline{\tau(\chi)}\,\chi(a)\,m(t, \psi\chi)
 \log x 
+ O(1)
\endaligned 
\end{equation}
as $x \to \infty$, for all integers $a$ with $1 \leq a < q_1$, $(a,q_1) = 1$, and $q_1 \mid q$,  
where $\chi_1 \bmod q_1$ is a Dirichlet character such that $\psi\chi_1$  
is the principal character modulo ${\rm lcm}(k,q)$.  
If such a $\chi_1$ exists, we set $\delta(\chi_1) = 1$;  
otherwise, we set $\delta(\chi_1) = 0$.
\item Let $\chi \bmod q$ ($\geq 1$) be a non-principal Dirichlet character  
such that $\tau(\chi) \ne 0$. 
Assuming the GRH for $L(s,\psi\chi)$, we have
\begin{equation} \label{EQ_602}
\sum_{a=1}^{q} \bar{\chi}(a)Z_\psi(x;t,a/q) 
 = \frac{e^{\pi i/4}\overline{\tau(\chi)}}{2\sqrt{2\pi}} \, m(t, \psi\chi) \log x
+ O(1)
\end{equation}
as $x \to \infty$. 
\end{enumerate}
\end{theorem}
\begin{proof} 
The proof is similar to those of Theorems~\ref{thm_1} and \ref{thm_0724}, so we only give a sketch.  
Instead of \eqref{EQ_203}, we use the identity obtained 
by multiplying both sides of \eqref{EQ_203} by $\psi(n)$:
\begin{equation} \label{EQ_603}
\psi(n) e(-an/q)
 = \frac{1}{\varphi(q)} \sum_{\chi} \overline{\tau(\chi)} \chi(a) \, \psi\chi(n),
\end{equation}
where the product $\psi \chi$ is a Dirichlet character modulo ${\rm lcm}(k,q)$.
In place of \eqref{EQ_109},  we define  
\[
\Psi_\psi(x; s, \alpha)
:=
\sum_{n \leq x} \frac{\Lambda(n)\psi(n)e(-\alpha n)}{n^s}
\left( 1 - \frac{\log n}{\log x} \right).
\]
Then, by \eqref{EQ_603},  
\begin{equation} \label{EQ_604}
\Psi_\psi(x; s, a/q)
= \frac{1}{\varphi(q)} \sum_\chi \overline{\tau(\chi)}\chi(a)
\Psi(x; s, \psi\chi).
\end{equation}
On the other hand, in place of the explicit formula \eqref{EQ_404}, 
we use  
\[
\sideset{}{'}\sum_{n \leq x}\frac{\Lambda(n)\psi(n)}{n^s}
=
- \frac{L^\prime}{L}(s,\psi) 
- \sum_{\rho_\psi} \frac{x^{\rho_\psi-s}}{\rho_\psi-s} 
+ \sum_{k=0}^{\infty} \frac{x^{-2k-\kappa(\psi)-s}}{2k+\kappa(\psi)+s}
\]
for $x>1$, $s \neq \rho_\chi, -2k-\kappa(\psi)$ ($k \in \Z_{\geq 0}$) \cite[(6)]{Y91}.  
Then, as an analogue of \eqref{EQ_401}, 
\begin{equation} \label{EQ_605}
\aligned 
\Psi_\psi(x;1/2+it,\alpha) 
& = - e^{-\pi i/4} \sqrt{2\pi} \, Z_\psi(x;t,\alpha)  
+O(1)
\endaligned 
\end{equation}
if $t$ is not the ordinate of any nontrivial zero of $L(s,\psi)$. 
Combining \eqref{EQ_604} and \eqref{EQ_605}, and using the explicit formulas for $\Psi(x; s, \chi)$ given in Section~\ref{section_1st_method}, we obtain \eqref{EQ_601}.  

Furthermore, by decomposing \eqref{EQ_601} via the orthogonality of characters, we have  
\[
\overline{\tau(\chi)}\Psi(x; s, \psi\chi)
= - e^{-\pi i/4} \sqrt{2\pi} \sum_{a=1}^{q} \overline{\chi}(a) Z_\psi(x; t, \alpha) + O(1).
\]
Therefore, as an analogue of \eqref{EQ_107} in Theorem~\ref{thm_0724}, we obtain \eqref{EQ_602}.
\end{proof}

Applying \eqref{EQ_602} to $\chi=\bar{\psi}$, we obtain  
\begin{equation} \label{EQ_606}
\sum_{a=1}^{q} \psi(a)Z_\psi(x;t,a/q) 
 = \frac{e^{\pi i/4}\overline{\tau(\bar{\psi})}}{2\sqrt{2\pi}} \, m(t, \chi_0 ) \log x
+ O(1), 
\end{equation}
where $\chi_0$ denotes the principal character modulo $k$. 
Since $m(t,\chi_0)$ is precisely the order of $\zeta(s)$ at $s = 1/2 + it$,  
\eqref{EQ_606} states that the nontrivial zeros of $\zeta(s)$ can be detected  
via the zeros of the Dirichlet $L$-function $L(s,\psi)$.  
The symmetry between \eqref{EQ_107} and \eqref{EQ_606} is evident. 
Furthermore, 
we can derive a relation between the nontrivial zeros of the Dirichlet $L$-functions 
associated with primitive Dirichlet characters $\psi \bmod q$ and $\chi \bmod q$ 
by substituting $\bar{\psi}\chi$ for $\chi$ in \eqref{EQ_606}. 
\medskip

\noindent
{\bf Acknowledgments}~
This work was supported by JSPS KAKENHI  Grant Number JP23K03050. 
%

%
\bigskip 

\noindent
Masatoshi Suzuki,\\[5pt]
Department of Mathematics, \\
Institute of Science Tokyo \\
2-12-1 Ookayama, Meguro-ku, \\
Tokyo 152-8551, Japan  \\[2pt]
Email: {\tt msuzuki@math.sci.isct.ac.jp}

\end{document}